\documentclass{amsart}
\usepackage{graphicx}
\newtheorem{thm}{Theorem}[section]

\newtheorem{lem}[thm]{Lemma}

\theoremstyle{definition}

\theoremstyle{remark} \theoremstyle{Proof}

\numberwithin{equation}{section}

\author[O. Ajebbar]{ Ajebbar Omar}
\address{Omar Ajebbar\\Department of Mathematics\\
Ibn Zohr University, Faculty of Sciences, Agadir\\
Morocco} \email{omar-ajb@hotmail.com}
\author[E. Elqorachi]{  Elqorachi Elqorachi}
\address{Elhoucien Elqorachi\\Department of Mathematics\\
Ibn Zohr University, Faculty of Sciences, Agadir\\
Morocco} \email{elqorachi@hotmail.com}
\begin{document}

\title[Stability of the Cosine-Sine Functional
equation]{Stability of the Cosine-Sine Functional Equation on
amenable groups}

\begin{abstract}
In this paper we establish the stability of the functional equation
\begin{equation*}f(xy)=f(x)g(y)+g(x)f(y)+h(x)h(y),\;x,y\in G,\end{equation*}
where $G$ is an amenable group.
\end{abstract}
\maketitle
\section{Introduction}
The stability problem of functional equations go back to 1940 when
Ulam \cite{Ulam} proposed a question concerning the stability of
group homomorphims. Hyers \cite{Hyers} gave a first affirmative
partial answer to the question of Ulam for Banach spaces. Hyers's
Theorem was generalized by Aoki \cite{Aoki} for additive mappings
and by Rassias \cite{Rassias} for linear mappings by considering an
unbounded Cauchy difference. The stability problem of several
functional equations have been extensively investigated by a number
of authors. An account on the further progress and developments in
this field can be found in \cite{Czerwik,Hyers et al.,Jung}.\\In
this paper we investigate the stability of the trigonometric
functional equation
\begin{equation}\label{eq:eq1}f(xy)=f(x)g(y)+g(x)f(y)+h(x)h(y),\,x,y\in G\end{equation}
on amenable groups.\\The continuous solutions of the trigonometric
functional equations
\begin{equation}\label{eq:eq2}f(xy)=f(x)g(y)+g(x)f(y),\,x,y\in G\end{equation}
and
\begin{equation}\label{eq:eq3}f(xy)=f(x)f(y)-g(x)g(y),\,x,y\in G\end{equation}
are obtained by Poulsen and Stetk{\ae}r \cite{Poulsen and Stetkaer},
where $G$ is a topological group that need not be abelian. Regular
solutions of (\ref{eq:eq2}) and (\ref{eq:eq3}) were described by
Acz\'{e}l \cite{Aczél} on abelian groups. Chung et al. \cite{Chung
et al.} solved the functional equation (\ref{eq:eq1}) on groups.
Recently, Ajebbar and Elqorachi \cite{Ajebbar and Elqorachi}
obtained the solutions of the functional equation (\ref{eq:eq1}) on
a semigroup generated by its squares. The stability properties of
the functional equations (\ref{eq:eq2}) and (\ref{eq:eq3}) have been
obtained by Sz\'{e}kelyhidi \cite{Székelyhidi3} on amenable groups.
\\The aim of the present paper is to extend the Sz\'{e}kelyhidi's
results \cite{Székelyhidi3} to the functional equation
(\ref{eq:eq1}).
\section{Definitions and preliminaries}
Throughout this paper $G$ denotes a semigroup (a set with an
associative composition) or a group. We denote by $\mathcal{B}(G)$
the linear space of all bounded complex-valued functions on $G$. We
call $a:G\rightarrow\mathbb{C}$ additive provided that
$a(xy)=a(x)+a(y)$ for all $x,y\in G$ and call
$m:G\rightarrow\mathbb{C}$ multiplicative provided that
$m(xy)=m(x)m(y)$ for all $x,y\in G$.\\
Let $\mathcal{V}$ be a linear space of complex-valued functions on
$G$. We say that the functions
$f_{1},\cdot\cdot\cdot,f_{n}:G\rightarrow\mathbb{C}$ are linearly
independent modulo $\mathcal{V}$ if
$\lambda_{1}\,f_{1}+\cdot\cdot\cdot+\lambda_{n}\,f_{n}\in
\mathcal{V}$ implies that
$\lambda_{1}=\cdot\cdot\cdot=\lambda_{n}=0$ for any
$\lambda_{1},\cdot\cdot\cdot,\lambda_{n}\in \mathbb{C}$. We say that
the linear space $\mathcal{V}$ is two-sided invariant if $f\in
\mathcal{V}$ implies that that the functions $x\mapsto f(xy)$ and
$x\mapsto f(yx)$ belong to $\mathcal{V}$ for any $y\in G$.\\
Notice that the linear space $\mathcal{B}(G)$ is two-sided
invariant.
\section{Basic results}
Throughout this section $G$ denotes a semigroup and $\mathcal{V}$ a
two-sided invariant linear space of complex-valued functions on $G$.
\begin{lem} Let $f,g,h:G\rightarrow\mathbb{C}$
be functions. Suppose that $f$, $g$ and $h$ are linearly independent
modulo $\mathcal{V}$. If the function
\begin{equation*}x\mapsto f(xy)-f(x)g(y)-g(x)f(y)-h(x)h(y)\end{equation*}
belongs to $\mathcal{V}$ for all $y\in G$, then there exist two
functions $\varphi_{1},\varphi_{2}\in \mathcal{V}$ such that
\begin{equation}\label{eq:eq4}\psi(x,y)=\varphi_{1}(x)f(y)+\varphi_{2}(x)h(y)\end{equation}
for all $x,y\in G$, where
\begin{equation}\label{eq:eq5}\psi(x,y):=f(xy)-f(x)g(y)-g(x)f(y)-h(x)h(y)\end{equation}
for $x,y\in G.$
\end{lem}
\begin{proof} We use a similar computation as the one of the proof
of \cite[Lemma 2.1]{Székelyhidi3}.\\Since the functions $f$, $g$ and
$h$ are linearly independent modulo  $\mathcal{V}$ so are $f$ and
$h$, then $f$ and $h$ are linearly independent. Then there exist
$y_{0}, z_{0}\in G$ such that
$f(y_{0})h(z_{0})-f(z_{0})h(y_{0})\neq0$, which implies that that
$f(y_{0})h(z_{0})\neq0$ or $f(z_{0})h(y_{0})\neq0$. We can finally
assume that $f(y_{0})\neq0$ and $h(z_{0})\neq0$. By applying
(\ref{eq:eq5}) to the pair $(x,y_{0})$ we derive
\begin{equation}\label{eq:eq6}g(x)=\alpha_{0}\,f(x)+\alpha_{1}\,h(x)+\alpha_{2}\,f(xy_{0})-\alpha_{2}\,\psi(x,y_{0})\end{equation}
for all $x\in G$, where
$\alpha_{0}:=-f(y_{0})^{-1}g(y_{0})\in\mathbb{C}$,
$\alpha_{1}:=-f(y_{0})^{-1}h(y_{0})\in\mathbb{C}$ and
$\alpha_{2}:=f(y_{0})^{-1}\in\mathbb{C}$ are constants. Similarly,
by applying (\ref{eq:eq5}) to pair $(x,z_{0})$, we get that
\begin{equation}\label{eq:eq7}h(x)=\beta_{0}\,f(x)+\beta_{1}\,g(x)+\beta_{2}\,f(xz_{0})-\beta_{2}\,\psi(x,z_{0})\end{equation}
for all $x\in G$, where
$\beta_{0}:=-h(z_{0})^{-1}g(z_{0})\in\mathbb{C}$,
$\beta_{1}:=-h(z_{0})^{-1}f(z_{0})\in\mathbb{C}$ and
$\beta_{2}:=h(z_{0})^{-1}\in\mathbb{C}$ are constants.\\Let $x\in G$
be arbitrary. Substituting (\ref{eq:eq7}) in (\ref{eq:eq6}) we
obtain
\begin{equation*}
\begin{split}g(x)&=\alpha_{0}\,f(x)+\alpha_{1}\,[\beta_{0}\,f(x)+\beta_{1}\,g(x)+\beta_{2}\,f(xz_{0})-\beta_{2}\,\psi(x,z_{0})]\\
&+\alpha_{2}\,f(xy_{0})-\alpha_{2}\,\psi(x,y_{0})\\
&=(\alpha_{0}+\alpha_{1}\,\beta_{0})\,f(x)+\alpha_{1}\,\beta_{1}\,g(x)+\alpha_{1}\,\beta_{2}\,f(xz_{0})-\alpha_{1}\,\beta_{2}\,\psi(x,z_{0})\\
&+\alpha_{2}\,f(xy_{0})-\alpha_{2}\,\psi(x,y_{0}).\end{split}\end{equation*}So
that
\begin{equation}\label{eq:eq8}\begin{split}(1-\alpha_{1}\beta_{1})\,g(x)&=(\alpha_{0}+\alpha_{1}\beta_{0})\,f(x)+\alpha_{1}\beta_{2}\,f(xz_{0})-\alpha_{1}\beta_{2}\,\psi(x,z_{0})\\
&+\alpha_{2}\,f(xy_{0})-\alpha_{2}\,\psi(x,y_{0}).\end{split}\end{equation}
Since $f(y_{0})h(z_{0})-f(z_{0})h(y_{0})\neq0$ and
$f(y_{0})h(z_{0})\neq0$ we get that $\alpha_{1}\,\beta_{1}\neq1$.
So, $x$ being arbitrary, we derive from (\ref{eq:eq8}) that there
exist $\gamma_{0},\,\gamma_{1},\,\gamma_{2}\in\mathbb{C}$ such that
\begin{equation}\label{eq:eq9}g(x)=\gamma_{0}\,f(x)+\gamma_{1}\,f(xy_{0})+\gamma_{2}\,f(xz_{0})-\gamma_{1}\,\psi(x,y_{0})-\gamma_{2}\,\psi(x,z_{0})\end{equation}
for all $x\in G$. Similarly we prove that there exist
$\delta_{0},\,\delta_{1},\,\delta_{2}\in\mathbb{C}$ such that
\begin{equation}\label{eq:eq10}h(x)=\delta_{0}\,f(x)+\delta_{1}\,f(xy_{0})+\delta_{2}\,f(xz_{0})-\delta_{1}\,\psi(x,y_{0})-\delta_{2}\,\psi(x,z_{0})\end{equation}
for all $x\in G$. Let $x,y,z\in G$ be arbitrary. In the following we
compute $f(xyz)$ first as $f((xy)z)$ and then as $f(x(yz))$. By
applying (\ref{eq:eq5}) to the pair $(xy,z)$, and taking
(\ref{eq:eq9}) and (\ref{eq:eq10}) into account, we obtain
\begin{equation*}
\begin{split}f((xy)z)&=f(xy)\,g(z)+g(xy)\,f(z)+h(xy)\,h(z)+\psi(xy,z)\\
&=[f(x)g(y)+g(x)\,f(y)+h(x)h(y)+\psi(x,y)]g(z)\\
&+[\gamma_{0}\,f(xy)+\gamma_{1}\,f(xyy_{0})+\gamma_{2}\,f(xyz_{0})-\gamma_{1}\,\psi(xy,y_{0})-\gamma_{2}\,\psi(xy,z_{0})]f(z)\\
&+[\delta_{0}\,f(xy)+\delta_{1}\,f(xyy_{0})+\delta_{2}\,f(xyz_{0})-\delta_{1}\,\psi(xy,y_{0})-\delta_{2}\,\psi(xy,z_{0})]h(z)\\
&+\psi(xy,z)\\
&=[f(x)g(y)+g(x)f(y)+h(x)h(y)+\psi(x,y)]g(z)\\
&+\gamma_{0}\,[f(x)g(y)+g(x)f(y)
+h(x)h(y)+\psi(x,y)]f(z)\\
&+\gamma_{1}\,[f(x)g(yy_{0})+g(x)f(yy_{0})+h(x)h(yy_{0})+\psi(x,yy_{0})]f(z)\\
&+\gamma_{2}\,[f(x)g(yz_{0})+g(x)f(yz_{0})+h(x)h(yz_{0})+\psi(x,yz_{0})]f(z)\\
&+\delta_{0}\,[f(x)g(y)+g(x)f(y)
+h(x)h(y)+\psi(x,y)]h(z)\\
&+\delta_{1}\,[f(x)g(yy_{0})+g(x)f(yy_{0})+h(x)h(yy_{0})+\psi(x,yy_{0})]h(z)\\
&+\delta_{2}\,[f(x)g(yz_{0})+g(x)f(yz_{0})+h(x)h(yz_{0})+\psi(x,yz_{0})]h(z)\\
&-[\gamma_{1}\,\psi(xy,y_{0})+\gamma_{2}\,\psi(xy,z_{0})]f(z)-[\delta_{1}\,\psi(xy,y_{0})+\delta_{2}\,\psi(xy,z_{0})]h(z)\\
&+\psi(xy,z).\end{split}\end{equation*} So that
\begin{equation}\label{eq:eq11}\begin{split}f((xy)z)&=f(x)[g(y)g(z)+\gamma_{0}\,g(y)f(z)+\gamma_{1}\,g(yy_{0})f(z)+\gamma_{2}\,g(yz_{0})f(z)\\
&+\delta_{0}\,g(y)h(z)
+\delta_{1}\,g(yy_{0})h(z)+\delta_{2}\,g(yz_{0})h(z)]\\
&+g(x)[f(y)g(z)+\gamma_{0}\,f(y)f(z)+\gamma_{1}\,f(yy_{0})f(z)
+\gamma_{2}\,f(yz_{0})f(z)\\
&+\delta_{0}\,f(y)h(z)
+\delta_{1}\,f(yy_{0})h(z)+\delta_{2}\,f(yz_{0})h(z)]\\
&+h(x)[h(y)g(z) +\gamma_{0}\,h(y)f(z)+\gamma_{1}\,h(yy_{0})f(z)
+\gamma_{2}\,h(yz_{0})f(z)\\
&+\delta_{0}\,h(y)h(z)+\delta_{1}\,h(yy_{0})h(z)
+\delta_{2}\,h(yz_{0})h(z)]\\
&+[\gamma_{0}\,\psi(x,y)+\gamma_{1}\,\psi(x,yy_{0})
+\gamma_{2}\,\psi(x,yz_{0})-\gamma_{1}\,\psi(xy,y_{0})\\
&-\gamma_{2}\,\psi(xy,z_{0})]f(z)
+\psi(x,y)g(z)+[\delta_{0}\,\psi(x,y)+\delta_{1}\,\psi(x,yy_{0})\\
&+\delta_{2}\,\psi(x,yz_{0})-\delta_{1}\,\psi(xy,y_{0})
-\delta_{2}\,\psi(xy,z_{0})]h(z)+\psi(xy,z).\end{split}\end{equation}
On the other hand, by applying (\ref{eq:eq5}) to the pair $(x,yz)$
we get that
\begin{equation}\label{eq:eq12}f(x(yz))=f(x)g(yz)+g(x)f(yz)+h(x)h(yz)+\psi(x,yz).\end{equation}
Now, let $y,\,z\in G$ be arbitrary. By assumption the functions
$$x\mapsto\psi(x,y),\,x\mapsto\psi(x,yy_{0}),\,x\mapsto\psi(x,yz_{0}),\,x\mapsto\psi(x,yz)$$
belong to $\mathcal{V}$. Moreover, since the linear space
$\mathcal{V}$ is two sided invariant the functions
$$x\mapsto\psi(xy,y_{0}),\,x\mapsto\psi(xy,z_{0}),\,x\mapsto\psi(xy,z)$$
belong to $\mathcal{V}$. Hence, by using (\ref{eq:eq11}),
(\ref{eq:eq12}) and the fact that $f$, $g$ and $h$ are linearly
independent modulo $\mathcal{V}$, we get that
\begin{equation}\label{eq:eq13}f(yz)=f(y)g(z)+[\gamma_{0}\,f(y)+\gamma_{1}\,f(yy_{0})+\gamma_{2}\,f(yz_{0})]f(z)\end{equation}
$$+[\delta_{0}\,f(y)+\delta_{1}\,f(yy_{0})+\delta_{2}\,f(yz_{0})]h(z).\quad\quad$$
From (\ref{eq:eq9}), (\ref{eq:eq10}) and (\ref{eq:eq13}) we get
\begin{equation*}\begin{split}f(yz)&=f(y)g(z)+[g(y)+\gamma_{1}\,\psi(y,y_{0})+\gamma_{2}\,\psi(y,z_{0})]f(z)\\
&+[h(y)+\delta_{1}\,\psi(y,y_{0})+\delta_{2}\,\psi(y,z_{0})]h(z)\\
&=f(y)g(z)+g(y)f(z)+h(y)h(z)+[\gamma_{1}\,\psi(y,y_{0})+\gamma_{2}\,\psi(y,z_{0})]f(z)\\
&+[\delta_{1}\,\psi(y,y_{0})+\delta_{2}\,\psi(y,z_{0})]h(z).\end{split}\end{equation*}
Hence, by using (\ref{eq:eq5}), we obtain
$$\psi(y,z)=[\gamma_{1}\,\psi(y,y_{0})+\gamma_{2}\,\psi(y,z_{0})]f(z)+[\delta_{1}\,\psi(y,y_{0})+\delta_{2}\,\psi(y,z_{0})]h(z).$$
So, $y$ and $z$ being arbitrary, we deduce (\ref{eq:eq4}) by putting
$$\varphi_{1}(x):=\gamma_{1}\,\psi(x,y_{0})+\gamma_{2}\,\psi(x,z_{0})$$
and
$$\varphi_{2}(x):=\delta_{1}\,\psi(x,y_{0})+\delta_{2}\,\psi(x,z_{0})$$
for all $x\in G$. This completes the proof of Lemma 3.1.\end{proof}
\begin{lem} Let $f,g,h:G\rightarrow\mathbb{C}$
be functions. Suppose that $f$ and $h$ are linearly independent
modulo $\mathcal{V}$ and $g\in \mathcal{V}$. If the function
\begin{equation*}x\mapsto f(xy)-f(x)g(y)-g(x)f(y)-h(x)h(y)\end{equation*}
belongs to $\mathcal{V}$ for all $y\in G$, then $g$ is
multiplicative.
\end{lem}
\begin{proof} Let $y,z\in G$ be arbitrary. By using the same computation as the one of the proof
of Lemma 3.1 we obtain from (\ref{eq:eq11}) and (\ref{eq:eq12}),
with the same notations, the following identity
\begin{equation*}\begin{split}&f(x)g(yz)+g(x)f(yz)+h(x)h(yz)+\psi(x,yz)\\
&=f(x)[g(y)g(z)+\gamma_{0}\,g(y)f(z)+\gamma_{1}\,g(yy_{0})f(z)+\gamma_{2}\,g(yz_{0})f(z)+\delta_{0}\,g(y)h(z)\\
&+\delta_{1}\,g(yy_{0})h(z)+\delta_{2}\,g(yz_{0})h(z)]+g(x)[f(y)g(z)+\gamma_{0}\,f(y)f(z)+\gamma_{1}\,f(yy_{0})f(z)\\
&+\gamma_{2}\,f(yz_{0})f(z)+\delta_{0}\,f(y)h(z)+\delta_{1}\,f(yy_{0})h(z)+\delta_{2}\,f(yz_{0})h(z)]+h(x)[h(y)g(z)\\
&+\gamma_{0}\,h(y)f(z)+\gamma_{1}\,h(yy_{0})f(z)+\gamma_{2}\,h(yz_{0})f(z)+\delta_{0}\,h(y)h(z)+\delta_{1}\,h(yy_{0})h(z)\\
&+\delta_{2}\,h(yz_{0})h(z)]+[\gamma_{0}\,\psi(x,y)+\gamma_{1}\,\psi(x,yy_{0})+\gamma_{2}\,\psi(x,yz_{0})
-\gamma_{1}\,\psi(xy,y_{0})\\
&-\gamma_{2}\,\psi(xy,z_{0})]
f(z)-\psi(x,y)g(z)+[\delta_{0}\,\psi(x,y)
+\delta_{1}\,\psi(x,yy_{0})+\delta_{2}\,\psi(x,yz_{0})\\
&-\delta_{1}\,\psi(xy,y_{0})
-\delta_{2}\,\psi(xy,z_{0})]h(z)+\psi(xy,z)
\end{split}\end{equation*}
for all $x\in G$. So that
\begin{equation}\label{eq:eq14}\begin{split}&f(x)[g(y)g(z)+\gamma_{0}\,g(y)f(z)+\gamma_{1}\,g(yy_{0})f(z)+\gamma_{2}\,g(yz_{0})f(z)+\delta_{0}\,g(y)h(z)\\
&+\delta_{1}\,g(yy_{0})h(z)+\delta_{2}\,g(yz_{0})h(z)-g(yz)]+h(x)[h(y)g(z)+\gamma_{0}\,h(y)f(z)\\
&+\gamma_{1}\,h(yy_{0})f(z)
+\gamma_{2}\,h(yz_{0})f(z)+\delta_{0}\,h(y)h(z)+\delta_{1}\,h(yy_{0})h(z)\\
&+\delta_{2}\,h(yz_{0})h(z)-h(yz)]\\
&=-g(x)[f(y)g(z)+\gamma_{0}\,f(y)f(z)+\gamma_{1}\,f(yy_{0})f(z)+\gamma_{2}\,f(yz_{0})f(z)+\delta_{0}\,f(y)h(z)\\
&+\delta_{1}\,f(yy_{0})h(z)+\delta_{2}\,f(yz_{0})h(z)-f(yz)]-[\gamma_{0}\,\psi(x,y)+\gamma_{1}\,\psi(x,yy_{0})\\
&+\gamma_{2}\,\psi(x,yz_{0})
-\gamma_{1}\,\psi(xy,y_{0})-\gamma_{2}\,\psi(xy,z_{0})]f(z)\\
&-[\delta_{0}\,\psi(x,y)+\delta_{1}\,\psi(x,yy_{0})+\delta_{2}\,\psi(x,yz_{0})-\delta_{1}\,\psi(xy,y_{0})
-\delta_{2}\,\psi(xy,z_{0})]h(z)\\
&-\psi(xy,z)+\psi(x,yz)\end{split}\end{equation} for all $x\in G$.
Since $g\in \mathcal{V}$, the function $x\mapsto\psi(x,t)$ belongs
to $\mathcal{V}$ for all $t\in G$ and $\mathcal{V}$ is a two
sided-invariant linear space of complex-valued functions on $G$, we
get that the right hand side of the identity (\ref{eq:eq14}) belongs
to $\mathcal{V}$ as a function in $x$, so does the left hand side of
(\ref{eq:eq14}). Since $f$ and $h$ are linearly independent modulo
$\mathcal{V}$, then we get that
\begin{equation}\begin{split}&g(y)g(z)+\gamma_{0}\,g(y)f(z)+\gamma_{1}\,g(yy_{0})f(z)+\gamma_{2}\,g(yz_{0})f(z)+\delta_{0}\,g(y)h(z)\\
&+\delta_{1}\,g(yy_{0})h(z)+\delta_{2}\,g(yz_{0})h(z)-g(yz)=0.\end{split}\end{equation}
So, $y$ and $z$ being arbitrary, then we get that
\begin{equation}\label{eq:eq15}\begin{split}g(yz)-g(y)g(z)&=[\gamma_{0}\,g(y)+\gamma_{1}\,g(yy_{0})+\gamma_{2}\,g(yz_{0})]f(z)\\
&+[\delta_{0}\,g(y)+\delta_{1}\,g(yy_{0})+\delta_{2}\,g(yz_{0})]h(z)\end{split}\end{equation}
for all $y,\,z\in G$. Now, let $y\in G$ be arbitrary. Since $g\in
\mathcal{V}$ and $\mathcal{V}$ is a two sided-invariant linear space
of complex-valued functions on $G$, we derive from (\ref{eq:eq15})
that the function
$$z\mapsto[\gamma_{0}\,g(y)+\gamma_{1}\,g(yy_{0})+\gamma_{2}\,g(yz_{0})]f(z)+[\delta_{0}\,g(y)+\delta_{1}\,g(yy_{0})+\delta_{2}\,g(yz_{0})]h(z)$$
belongs to $\mathcal{V}$. Hence, seeing that $f$ and $h$ are
linearly independent modulo $\mathcal{V}$, we get that
$\gamma_{0}\,g(y)+\gamma_{1}\,g(yy_{0})+\gamma_{2}\,g(yz_{0})=0$ and
$\delta_{0}\,g(y)+\delta_{1}\,g(yy_{0})+\delta_{2}\,g(yz_{0})=0$.
Substituting this back into (\ref{eq:eq15}) we obtain
$g(yz)=g(y)g(z)$ for all $z\in G$. So, $y$ being arbitrary, we
deduce that $g$ is multiplicative. This completes the proof of Lemma
3.2.\end{proof}
\begin{lem} Let $f,g,h:G\rightarrow\mathbb{C}$
be functions. Suppose that $f$ and $h$ are linearly dependent modulo
$\mathcal{V}$. If the function
\begin{equation*}x\mapsto f(xy)-f(x)g(y)-g(x)f(y)-h(x)h(y)\end{equation*}
belongs to $\mathcal{V}$ for all $y\in G$, then we have one of the
following possibilities:\\
(1) $f=0$, $g$ is arbitrary and $h\in \mathcal{V}$;\\
(2) $f,g,h\in \mathcal{V}$;\\
(3) $g+\frac{\lambda^{2}}{2}\,f=m-\lambda\,\varphi$,
$h-\lambda\,f=\varphi$, where $\lambda\in\mathbb{C}$ is a constant,
$\varphi\in \mathcal{V}$ and $m:G\rightarrow\mathbb{C}$ is a
multiplicative function such that
$m\in\mathcal{V}$;\\
(4) $f=\alpha\,m-\alpha\,b$,
$g=\frac{1-\alpha\,\lambda^{2}}{2}\,m+\frac{1+\alpha\,\lambda^{2}}{2}\,b-\lambda\,\varphi$,
$h=\alpha\lambda\,m-\alpha\lambda\,b+\varphi$, where
$\alpha,\lambda\in\mathbb{C}$ are constants,
$m:G\rightarrow\mathbb{C}$ is a multiplicative function and
$b,\varphi\in \mathcal{V}$;\\
(5) $f=f_{0}$,
$g=g_{0}-\frac{\lambda^{2}}{2}\,f_{0}-\lambda\,\varphi$,
$h=\lambda\,f_{0}+\varphi$, where $\lambda\in\mathbb{C}$ is a
constant, $\varphi\in\mathcal{V}$ and
$f_{0},g_{0}:G\rightarrow\mathbb{C}$ satisfy the sine functional
equation
$$f_{0}(xy)=f_{0}(x)g_{0}(y)+g_{0}(x)f_{0}(y),\,\,x,y\in G.$$
\end{lem}
\begin{proof} Let $\psi$ be the function defined in
(\ref{eq:eq5}). If $f=0$ then $g$ is arbitrary and the function
$x\mapsto h(x)h(y)$ belongs to $\mathcal{V}$ for all $y$ in $G$.
Hence $h\in\mathcal{V}$. The result occurs in $(1)$ of Lemma 3.3. In
what follows we assume that $f\neq0$. We have the following
cases\\
\underline{Case 1}: $h\in\mathcal{V}$. Then the function $x\mapsto
h(x)h(y)$ belongs to $\mathcal{V}$ for all $y$ in $G$. So that the
function $x\mapsto f(xy)-f(x)g(y)-g(x)f(y)$ belongs to $\mathcal{V}$
for all $y$ in $G$. So, according to \cite[Lemma 2.2]{Székelyhidi3}
and taking into account that $f\neq0$, we get that one of the following possibilities holds\\
(i) $f,g,h\in \mathcal{V}$ which occurs in (2) of Lemma
3.3.\\
(ii) $g=m$ and $h=\varphi$, where $\varphi\in \mathcal{V}$ and
$m:G\rightarrow\mathbb{C}$ is a multiplicative function such that
$m\in\mathcal{V}$. This is the result (3) of Lemma 3.3 for
$\lambda=0$.\\
(iii) $f=\alpha\,m-\alpha\,b$, $g=\frac{1}{2}m+\frac{1}{2}b$,
$h=\varphi$,where $\alpha\in\mathbb{C}$ is a constant,
$m:G\rightarrow\mathbb{C}$ is a multiplicative function and
$b,\varphi\in \mathcal{V}$. This is the result (4) of Lemma 3.3 for
$\lambda=0$.\\
(iv) $f(xy)=f(x)g(y)+g(x)f(y)$ for all $x,y\in G$ and $h=\varphi$,
where $\varphi\in\mathcal{V}$, which is the result (5) of Lemma 3.3 for $\lambda=0$.\\
\underline{Case 2}: $h \not\in\mathcal{V}$. If $f\in\mathcal{V}$
then the function $x\mapsto f(xy)$ belongs to $\mathcal{V}$ for all
$y\in G$, because the linear space $\mathcal{V}$ is two-sided
invariant. As the function $x\mapsto \psi(x,y)$ belongs to
$\mathcal{V}$ for all $y\in G$ we get that the function $x\mapsto
g(x)f(y)+h(x)h(y)$ belongs to $\mathcal{V}$ for all $y\in G$. Since
$h\not \in\mathcal{V}$ we have $h\neq0$. We derive that there exist
a constant $\alpha\in\mathbb{C}\setminus\{0\}$ and a function
$k\in\mathcal{V}$ such that
\begin{equation}\label{eq:eq16}h=\alpha\,g+k,\end{equation}
so that
\begin{equation*}\begin{split}\psi(x,y)&=f(xy)-f(x)g(y)-g(x)f(y)-(\alpha\,g(x)+k(x))(\alpha\,g(y)+k(y))\\
&=f(xy)-f(x)g(y)-g(x)f(y)-\alpha^{2}\,g(x)g(y)-\alpha\,g(x)k(y)-\alpha\,k(x)g(y)\\
&-k(x)k(y)\\
&=f(xy)-k(x)k(y)-g(x)[f(y)+\alpha^{2}\,g(y)+\alpha\,k(y)]-g(y)[f(x)+\alpha\,k(x)]\\
&=f(xy)-k(x)k(y)-g(x)[f(y)+\alpha\,h(y)]-g(y)[f(x)+\alpha\,k(x)]\end{split}\end{equation*}
for all $x,y\in G$. Since the functions $x\mapsto f(xy)$, $x\mapsto
k(x)k(y)$, $x\mapsto g(y)[f(x)+\alpha\,k(x)]$ and $x\mapsto
\psi(x,y)$ belong to $\mathcal{V}$ for all $y\in G$, we derive from
the identity above that the function $x\mapsto
g(x)[f(y)+\alpha\,h(y)]$ belongs to $\mathcal{V}$ for all $y\in G$,
which implies that that $g\in\mathcal{V}$ or $f(y)+\alpha\,h(y)=0$
for all $y\in G$. Hence, since $\alpha\in\mathbb{C}\setminus\{0\}$,
we get that $g\in\mathcal{V}$ or $h=-\dfrac{1}{\alpha}\,f$. So,
taking (\ref{eq:eq16}) into account, we get that $h\in\mathcal{V}$;
which contradicts the assumption on $h$, hence
$f\not\in\mathcal{V}$. As $f$ and $h$ are linearly dependent modulo
$\mathcal{V}$ we infer that there exist a constant
$\lambda\in\mathbb{C}\setminus\{0\}$ and a function
$\varphi\in\mathcal{V}$ such that
\begin{equation}\label{eq:eq17}h=\lambda\,f+\varphi.\end{equation}
So we get from (\ref{eq:eq5}) that
\begin{equation*}\begin{split}\psi(x,y)&=f(xy)-f(x)g(y)-g(x)f(y)-(\lambda\,f(x)+\varphi(x))(\lambda\,f(y)+\varphi(y))\\
&=f(xy)-f(x)g(y)-g(x)f(y)-\lambda^{2}\,f(x)f(y)-\lambda\,f(x)\varphi(y)-\lambda\,\varphi(x)f(y)\\
&-\varphi(x)\varphi(y)\\
&=f(xy)-\varphi(x)\varphi(y)-f(x)[g(y)+\frac{\lambda^{2}}{2}\,f(y)+\lambda\,\varphi(y)]\\
&-[g(x)+\frac{\lambda^{2}}{2}\,f(x)+\lambda\,\varphi(x)]f(y),\end{split}\end{equation*}
for all $x,y\in G$, which implies that that
\begin{equation}\label{eq:eq18}\psi(x,y)+\varphi(x)\varphi(y)=f(xy)-f(x)\phi(y)-\phi(x)f(y)\end{equation}
for all $x,y\in G$, where
\begin{equation}\label{eq:eq19}\phi:=g+\frac{\lambda^{2}}{2}\,f+\lambda\,\varphi.\end{equation}
Since $\varphi\in\mathcal{V}$ and the function $x\mapsto\psi(x,y)$
belongs to $\mathcal{V}$ for all $y\in G$ we get from
(\ref{eq:eq18}) that the function $$x\mapsto
f(xy)-f(x)\phi(y)-\phi(x)f(y)$$ belongs to $\mathcal{V}$ for all
$y\in G$. Moreover $\mathcal{V}$ is a two-sided invariant linear
space of complex-valued function. Hence, according to \cite[Lemma
2.2]{Székelyhidi3} and taking into account that $f,h\not\in
\mathcal{V}$, we have one of the following
possibilities:\\
(i) $\phi=m$ where $m\in\mathcal{V}$ is multiplicative. Then we get,
from (\ref{eq:eq19}) and (\ref{eq:eq17}), that
$g+\frac{\lambda^{2}}{2}\,f=m-\lambda\,\varphi$ and
$h-\lambda\,f=\varphi$, where $\varphi\in\mathcal{V}$. The result
occurs in (3) of Lemma 3.3.\\
(ii) $f=\alpha\,m-\alpha\,b$, $\phi=\frac{1}{2}m+\frac{1}{2}b$,
where $m:G\rightarrow\mathbb{C}$ is multiplicative,
$b:G\rightarrow\mathbb{C}$ is in $\mathcal{V}$ and
$\alpha\in\mathbb{C}$ is a constant. Taking (\ref{eq:eq19}) and
(\ref{eq:eq17}) into account, we obtain respectively
$$g=\frac{1}{2}m+\frac{1}{2}b-\frac{\lambda^{2}}{2}\,(\alpha\,m-\alpha\,b)-\lambda\,\varphi$$
$$=\frac{1-\alpha\,\lambda^{2}}{2}m+\frac{1+\alpha\,\lambda^{2}}{2}b-\lambda\,\varphi\quad$$
and
$$h=\alpha\lambda\,m-\alpha\lambda\,b+\varphi.$$
So the result (4) of Lemma 3.3 holds.\\
(iii) $f(xy)=f(x)\phi(y)+\phi(x)f(y)$ for all $x,y\in G$. The result
(5) of Lemma 3.3 holds easily by using the identities
(\ref{eq:eq17}) and (\ref{eq:eq19}). This completes the proof of
Lemma 3.3.\end{proof}
\begin{lem} Let $f,g,h:G\rightarrow\mathbb{C}$
be functions. Suppose that $f$ and $h$ are linearly independent
modulo $\mathcal{V}$. If the functions
\begin{equation*}x\mapsto f(xy)-f(x)g(y)-g(x)f(y)-h(x)h(y)\end{equation*}
and
\begin{equation*}x\mapsto f(xy)-f(yx)\end{equation*}
belong to $\mathcal{V}$ for all $y\in G$, then we have one of the
following possibilities:\\
(1) $f=-\lambda^{2}\,f_{0}+\lambda^{2}\,\varphi$,
$g=\frac{1+\rho^{2}}{2}\,f_{0}+\rho\,g_{0}+\frac{1-\rho^{2}}{2}\,\varphi$,
$h=\lambda\rho\,f_{0}+\lambda\,g_{0}-\lambda\rho\,\varphi$, where
$\lambda\in\mathbb{C}\setminus\{0\},\rho\in\mathbb{C}$ are
constants, $\varphi\in \mathcal{V}$ and
$f_{0},g_{0}:G\rightarrow\mathbb{C}$ satisfy the cosine functional
equation
$$f_{0}(xy)=f_{0}(x)f_{0}(y)-g_{0}(x)g_{0}(y)$$ for all $x,y\in
G$;\\
(2)
\begin{equation*}\begin{split}f(xy)-\lambda^{2}\,M(xy)&=(f(x)-\lambda^{2}\,M(x))m(y)+m(x)(f(y)
-\lambda^{2}\,M(y))\\&+\lambda^{2}\,m(xy)+\psi(x,y)\end{split}\end{equation*}
for all $x,y\in G,$$$g=\frac{1}{2}\beta^{2}\,f+\beta\,h+m$$ and
$$\beta\,f+h=\lambda\,M-\lambda\,m,$$ where
$\beta\in\mathbb{C},\,\lambda\in\mathbb{C}\setminus\{0\}$ are
constants, $m,M:G\rightarrow\mathbb{C}$ are multiplicative functions
such that $m\in\mathcal{V}$, $M\not\in\mathcal{V}$ and $\psi$ is the function defined in (\ref{eq:eq5});\\
(3) $$f(xy)=f(x)m(y)+m(x)f(y)+H(x)H(y)+\psi(x,y),$$
$$g=\frac{1}{2}\beta^{2}\,f+\beta\,h+m$$
and
$$H(xy)-m(x)H(y)-H(x)m(y)=\eta_{1}\,\psi(x,y)+\eta_{2}\,m(x)L_{1}(y)+\eta_{3}\,m(x)L_{2}(y)$$
$$+\eta_{4}\,\psi(x,l_{1}(y))+\eta_{5}\,\psi(x,l_{2}(y))+\eta_{6}\,L_{1}(xy)+\eta_{7}\,L_{2}(xy)$$
for all $x,y\in G$, where
$\beta,\eta_{1},\cdot\cdot\cdot,\eta_{7}\in\mathbb{C}$ are
constants, $m:G\rightarrow\mathbb{C}$ is a multiplicative function
in $\mathcal{V}$, $L_{1},L_{2}\in\mathcal{V}$,
$l_{1},l_{2}:G\rightarrow G$
are mappings, $H=\beta\,f+h$ and $\psi$ is the function defined in (\ref{eq:eq5});\\
(4) $f(xy)=f(x)g(y)+g(x)f(y)+h(x)h(y)$ for all $x,y\in G$.
\end{lem}
\begin{proof} We split the discussion into the cases $f,g,h$ are
linearly dependent modulo $\mathcal{V}$ and $f,g,h$ are linearly
independent modulo $\mathcal{V}$.\\
\underline{Case A}: $f,g,h$ are linearly dependent modulo
$\mathcal{V}$. Since $f$ and $h$ are linearly independent modulo
$\mathcal{V}$ we get that there exist a function
$\varphi\in\mathcal{V}$ and two constants
$\alpha,\beta\in\mathbb{C}$ such that
\begin{equation}\label{eq:eq20}g=\alpha\,f+\beta\,h+\varphi.\end{equation}
By substituting (\ref{eq:eq20}) in (\ref{eq:eq5}) we obtain
\begin{equation*}\begin{split}\psi(x,y)&=f(xy)-f(x)[\alpha\,f(y)+\beta\,h(y)+\varphi(y)]-[\alpha\,f(x)+\beta\,h(x)+\varphi(x)]f(y)\\
&-h(x)h(y)\\
&=f(xy)-2\,\alpha\,f(x)f(y)-f(x)\varphi(y)-\varphi(x)f(y)-\beta\,f(x)h(y)-\beta\,h(x)f(y)\\
&-h(x)h(y),\end{split}\end{equation*} for all $x,y\in G$, which
implies that
\begin{equation}\label{eq:eq21}\begin{split}\psi(x,y)&=f(xy)-(2\,\alpha\,-\beta^{2})f(x)f(y)-f(x)\varphi(y)-\varphi(x)f(y)\\
&-[\beta\,f(x)+h(x)][\beta\,f(y)+h(y)]\end{split}\end{equation}
for all $x,y\in G$. We have the following subcases\\
\underline{Subcase A.1}: $2\,\alpha\neq\beta^{2}$. Let $x,y\in G$ be
arbitrary and let $\delta\in\mathbb{C}\setminus\{0\}$ such that
\begin{equation}\label{eq:eq22}\delta^{2}=-(2\,\alpha\,-\beta^{2}).\end{equation}
Multiplying both sides of (\ref{eq:eq21}) by $-\delta^{2}$ and then
adding $\varphi(xy)-\varphi(x)\varphi(y)$ to both sides of the
identity obtained we derive
\begin{equation*}\begin{split}&-\delta^{2}\,\psi(x,y)+\varphi(xy)-\varphi(x)\varphi(y)=-\delta^{2}\,f(xy)+\varphi(xy)-[\delta^{4}\,f(x)f(y)\\
&-\delta^{2}\,f(x)\varphi(y)-\delta^{2}\,\varphi(x)f(y)+\varphi(x)\varphi(y)]+\delta^{2}\,[\beta\,f(x)+h(x)][\beta\,f(y)+h(y)].\end{split}\end{equation*}
So, $x$ and $y$ being arbitrary, we get from the identity above that
\begin{equation}\label{eq:eq23}-\delta^{2}\,\psi(x,y)+\varphi(xy)-\varphi(x)\varphi(y)=f_{0}(xy)-f_{0}(x)f_{0}(y)+g_{0}(x)g_{0}(y),\end{equation}
for all $x,y\in G$, where
\begin{equation}\label{eq:eq24}f_{0}:=-\delta^{2}\,f+\varphi\end{equation}
and
\begin{equation}\label{eq:eq25}g_{0}:=\delta\,(\beta\,f+h).\end{equation}
Notice that $f_{0}$ and $g_{0}$ are linearly independent modulo
$\mathcal{V}$ because $f$ and $h$ are.\\
Now, let $y$ be arbitrary. As $\varphi\in\mathcal{V}$ the function
$x\mapsto\varphi(x)\varphi(y)$ belongs to $\mathcal{V}$, and since
the linear space $\mathcal{V}$ is two-sided invariant, we get that
the function $x\mapsto\varphi(xy)$ belongs to $\mathcal{V}$.
Moreover, by assumption the function $x\mapsto \psi(x,y)$ belongs to
$\mathcal{V}$. Hence the left hand side of the identity
(\ref{eq:eq23}) belongs to $\mathcal{V}$ as a function in $x$. So
that the function
$$x\mapsto f_{0}(xy)-f_{0}(x)f_{0}(y)+g_{0}(x)g_{0}(y)$$
belongs to $\mathcal{V}$. On the other hand, by using
(\ref{eq:eq24}), we have
$$f_{0}(xy)-f_{0}(yx)=-\delta^{2}\,(f(xy)-f(yx))+\varphi(xy)-\varphi(yx)$$ for all $x\in G$.
So, $y$ being arbitrary, the function $x\mapsto f_{0}(xy)-f_{0}(yx)$
belongs to $\mathcal{V}$ for all $y\in G$ because the functions
$x\mapsto f(xy)-f(yx)$ and $x\mapsto \varphi(xy)-\varphi(yx)$ do.
Moreover $f_{0}$ and $g_{0}$ are linearly independent modulo
$\mathcal{V}$. Hence we get, according to \cite[Lemma
3.1]{Székelyhidi3}, that
$$f_{0}(xy)=f_{0}(x)f_{0}(y)-g_{0}(x)g_{0}(y)$$ for all $x,y\in G$.
By putting $\lambda=\dfrac{1}{\delta}$ we get, from (\ref{eq:eq24}),
that
\begin{equation}\label{eq:eq26}f=-\lambda^{2}\,f_{0}+\lambda^{2}\varphi.\end{equation}
By putting $\rho=\beta\lambda$ we get, from (\ref{eq:eq25}), that
$h=\lambda\,g_{0}-\beta\,(-\lambda^{2}\,f_{0}+\lambda^{2}\varphi)$,
which implies that
\begin{equation}\label{eq:eq27}h=\lambda\rho\,f_{0}+\lambda\,g_{0}-\lambda\rho\,\varphi.\end{equation}
So, we derive from (\ref{eq:eq20}), (\ref{eq:eq26}) and
(\ref{eq:eq27}) that
\begin{equation*}\begin{split}g&=\alpha\,(-\lambda^{2}\,f_{0}+\lambda^{2}\,\varphi)+\beta\,(\lambda\rho\,f_{0}+\lambda\,g_{0}-\lambda\rho\,\varphi)+\varphi\\
&=(-\alpha\lambda^{2}+\beta\lambda\rho)\,f_{0}+\beta\lambda\,g_{0}+(\alpha\lambda^{2}-\beta\lambda\rho+1)\,\varphi\\
&=(-\alpha\lambda^{2}+\rho^{2})\,f_{0}+\rho\,g_{0}+(\alpha\lambda^{2}-\rho^{2}+1)\,\varphi\end{split}\end{equation*}
Using (\ref{eq:eq22}) we find, by elementary computations, that
$\alpha\lambda^{2}=\frac{1}{2}\,\rho^{2}-\frac{1}{2}$. Hence, from
the identity above, we get that
\begin{equation*}g=\frac{1+\rho^{2}}{2}\,f_{0}+\rho\,g_{0}+\frac{1-\rho^{2}}{2}\,\varphi.\end{equation*}
The result obtained in this case occurs in (1) of Lemma 3.4.\\
\underline{Subcase A.2}: $2\,\alpha=\beta^{2}$. In this case the
identity (\ref{eq:eq21}) becomes
\begin{equation}\label{eq:eq28}\psi(x,y)=f(xy)-f(x)\varphi(y)-\varphi(x)f(y)-H(x)H(y)\end{equation}
for all $x,y\in G$, where
\begin{equation}\label{eq:eq29}H:=\beta\,f+h.\end{equation}
Since $f$ and $h$ are linearly independent modulo $\mathcal{V}$ so
are $f$ and $H$. Moreover $\varphi\in\mathcal{V}$. Hence, according
to Lemma 3.2, there exists a multiplicative function
$m:G\rightarrow\mathbb{C}$ in $\mathcal{V}$ such that $\varphi=m$.
So the identities (\ref{eq:eq20}) and (\ref{eq:eq28}) become
respectively
\begin{equation}\label{eq:eq30}g=\frac{1}{2}\beta^{2}\,f+\beta\,h+m.\end{equation}
and
\begin{equation}\label{eq:eq31}\psi(x,y)=f(xy)-f(x)m(y)-m(x)f(y)-H(x)H(y)\end{equation}
for all $x,y\in G$. We use similar computations to the ones in the
proof of \cite[Theorem]{Chung et al.}. Let $x,y,z\in G$ be
arbitrary. First we compute $f(xyz)$ as $f(x(yz))$ and then as
$f((xy)z)$. From (\ref{eq:eq31}) we get that
\begin{equation*}\begin{split}f(x(yz))&=f(x)m(yz)+m(x)f(yz)+H(x)H(yz)+\psi(x,yz)\\
&=f(x)m(yz)+m(x)[f(y)m(z)+m(y)f(z)+H(y)H(z)+\psi(y,z)]\\
&+H(x)H(yz)+\psi(x,yz),\end{split}\end{equation*} so that
\begin{equation}\label{eq:eq32}\begin{split}f(x(yz))&=f(x)m(yz)+m(xz)f(y)+m(xy)f(z)+m(x)H(y)H(z)\\
&+m(x)\psi(y,z)+H(x)H(yz)+\psi(x,yz).\end{split}\end{equation} On
the other hand
\begin{equation*}\begin{split}f((xy)z)&=f(xy)m(z)+m(xy)f(z)+H(xy)H(z)+\psi(xy,z)\\
&=[f(x)m(y)+m(x)f(y)+H(x)H(y)+\psi(x,y)]m(z)+m(xy)f(z)\\
&+H(xy)H(z)+\psi(xy,z),\end{split}\end{equation*} hence
\begin{equation}\label{eq:eq33}\begin{split}f((xy)z)&=f(x)m(yz)+m(xz)f(y)+m(xy)f(z)+H(x)H(y)m(z)\\
&+m(z)\psi(x,y)+H(xy)H(z)+\psi(xy,z).\end{split}\end{equation} From
(\ref{eq:eq32}) and (\ref{eq:eq33}) we get that
\begin{equation}\label{eq:eq34}\begin{split}&H(x)[H(yz)-H(y)m(z)-m(y)H(z)]-H(z)[H(xy)-m(x)H(y)-H(x)m(y)]\\
&=m(z)\psi(x,y)-m(x)\psi(y,z)+\psi(xy,z)-\psi(x,yz),\end{split}\end{equation}
for all $x,y,z\in G$. Since $f$ and $H$ are linearly independent
modulo $\mathcal{V}$ they are, in particular, linearly independent.
So, there exist $z_{1},z_{2}\in G$ such that
\begin{equation}\label{eq:eq35}f(z_{1})H(z_{2})-f(z_{2})H(z_{1})\neq0.\end{equation}
Let $x,y\in G$ be arbitrary. By putting $z=z_{1}$ and then $z=z_{2}$
in (\ref{eq:eq34}) we get respectively
\begin{equation}\label{eq:eq36}H(x)k_{i}(y)-H(z_{i})[H(xy)-H(x)m(y)-m(x)H(y)]=\psi_{i}(x,y)\end{equation}
where
$$k_{i}(y):=H(yz_{i})-H(y)m(z_{i})-m(y)H(z_{i})$$
and
\begin{equation}\label{eq:eq37}\psi_{i}(x,y):=m(z_{i})\psi(x,y)-m(x)\psi(y,z_{i})-\psi(x,yz_{i})+\psi(xy,z_{i})\end{equation}
for $i=1,2$. Multiplying both sides of (\ref{eq:eq36}) by $f(z_{2})$
for $i=1$ and by $f(z_{1})$ for $i=2$, and subtracting the
identities obtained we get that
\begin{equation}\label{eq:eq38}H(x)k_{3}(y)+[f(z_{1})H(z_{2})-f(z_{2})H(z_{1})][H(xy)-H(x)m(y)-m(x)H(y)]=\psi_{3}(x,y),\end{equation}
where
$$k_{3}(y):=f(z_{2})k_{1}(y)-f(z_{1})k_{2}(y)$$
and
\begin{equation}\label{eq:eq39}\psi_{3}(x,y):=f(z_{2})\psi_{1}(x,y)-f(z_{1})\psi_{2}(x,y).\end{equation}
So, $x$ and $y$ being arbitrary, we get, taking (\ref{eq:eq35}) and
(\ref{eq:eq38}) into account, that
\begin{equation}\label{eq:eq40}H(xy)-H(x)m(y)-m(x)H(y)=H(x)k(y)+\Phi(x,y)\end{equation}
for all $x,y\in G$, where
$$k(x):=-[f(z_{1})H(z_{2})-f(z_{2})H(z_{1})]^{-1}k_{3}(x)$$
and
\begin{equation}\label{eq:eq41}\Phi(x,y):=[f(z_{1})H(z_{2})-f(z_{2})H(z_{1})]^{-1}\psi_{3}(x,y)\end{equation}
for all $x,y\in G.$ Substituting (\ref{eq:eq40}) into
(\ref{eq:eq34}) we get that
\begin{equation*}\begin{split}&H(x)[H(y)k(z)+\Phi(y,z)]-H(z)[H(x)k(y)+\Phi(x,y)]\\
&=m(z)\psi(x,y)-m(x)\psi(y,z)+\psi(xy,z)-\psi(x,yz),\end{split}\end{equation*}
which implies that
\begin{equation}\label{eq:eq42}\begin{split}H(x)[H(y)k(z)-H(z)k(y)+\Phi(y,z)]&=H(z)\Phi(x,y)+m(z)\psi(x,y)\\
&-m(x)\psi(y,z)+\psi(xy,z)-\psi(x,yz)\end{split}\end{equation} for
all $x,y,z\in G$. Now let $y,z\in G$ be arbitrary. Since
$\mathcal{V}$ is a two-sided invariant linear space of
complex-valued functions on $G$, and the functions $x\mapsto m(x)$
and $x\mapsto \psi(x,y)$ belong to $\mathcal{V}$, we deduce from
(\ref{eq:eq37}), (\ref{eq:eq39}) and (\ref{eq:eq41}) that the
functions $x\mapsto\Phi(x,y)$ and $x\mapsto \psi_{i}(x,y)$ belong to
$\mathcal{V}$ for $i=1,2,3$. Hence the right hand side of
(\ref{eq:eq42}) belongs to $\mathcal{V}$ as a function in $x$. It
follows that the left hand side of (\ref{eq:eq42}) belongs to
$\mathcal{V}$ as a function in $x$. As $f$ and $H$ are linearly
independent modulo $\mathcal{V}$, we derive, from (\ref{eq:eq42}),
that $H(y)k(z)-H(z)k(y)+\Phi(y,z)=0$. So, $y$ and $z$ being
arbitrary, we get that
\begin{equation}\label{eq:eq43}H(z)k(x)=H(x)k(z)+\Phi(x,z)\end{equation}
for all $x,z\in G.$\\
On the other hand we deduce from (\ref{eq:eq35}) that
$f(z_{1})H(z_{2})\neq0$ or $f(z_{2})H(z_{1})\neq0$, so we can
assume, without loss of generality, that $H(z_{1})\neq0$. Replacing
$z$ by $z_{1}$ in the identity (\ref{eq:eq43}) we derive that
\begin{equation}\label{eq:eq44}k(x)=\gamma\,H(x)+\Phi_{1}(x)\end{equation}
for all $x\in G$, where $\gamma:=H(z_{1})^{-1}k(z_{1})$ and
\begin{equation}\label{eq:eq45}\Phi_{1}(x):=H(z_{1})^{-1}\Phi(x,z_{1})\end{equation}
for all $x\in G$. From (\ref{eq:eq40}) and (\ref{eq:eq44}) we get
that
\begin{equation}\label{eq:eq46}H(xy)=H(x)m(y)+m(x)H(y)+\gamma\,H(x)H(y)+H(x)\Phi_{1}(y)+\Phi(x,y)\end{equation}
for all $x,y\in G$. Since the functions $m$ and $x\mapsto \Phi(x,y)$
belongs to $\mathcal{V}$ for all $y\in G$ we get, from
(\ref{eq:eq46}), that the function
\begin{equation}\label{eq:eq47}x\mapsto
H(xy)-H(x)[m(y)+\Phi_{1}(y)+\gamma\,H(y)]\end{equation} belongs to
$\mathcal{V}$ for all $y\in G$. As $H\not\in\mathcal{V} $ we get
from (\ref{eq:eq47}), according to \cite[Theorem]{Székelyhidi2},
that there exists a multiplicative function
$M:G\rightarrow\mathbb{C}$ such that
\begin{equation}\label{eq:eq48}m+\Phi_{1}+\gamma\,H=M.\end{equation}
We have the following subcases\\
\underline{Case A.2.1}: $\gamma\neq0$. Putting
$\lambda=\dfrac{1}{\gamma}\in\mathbb{C}\setminus\{0\}$ we obtain
from (\ref{eq:eq48}) the identity
\begin{equation}\label{eq:eq49}H=\lambda\,M-\lambda\,m-\lambda\,\Phi_{1}.\end{equation}
Let $x,y\in G$ be arbitrary. Since $m$ and $M$ are multiplicative we
get from the identity above that
$H(xy)-H(yx)=\lambda\,\Phi_{1}(yx)-\lambda\,\Phi_{1}(xy)$. Taking
(\ref{eq:eq46}) into account we get that
$H(x)\Phi_{1}(y)-H(y)\Phi_{1}(x)+\Phi(x,y)-\Phi(y,x)=\lambda\,\Phi_{1}(yx)-\lambda\,\Phi_{1}(xy)$.
So, $x$ and $y$ being arbitrary, we obtain
\begin{equation}\label{eq:eq50}H(x)\Phi_{1}(y)=H(y)\Phi_{1}(x)+\Phi(y,x)-\Phi(x,y)+\lambda\,\Phi_{1}(yx)-\lambda\,\Phi_{1}(xy)\end{equation}
for all $x,y\in G$. Now let $y$ be arbitrary. As seen early the
functions $\Phi_{1}$ and $x\mapsto\Phi(x,y)-\Phi(y,x)$ belong to
$\mathcal{V}$. So, $\mathcal{V}$ being a tow-sided invariant linear
space of complex-valued functions on $G$, we get from
(\ref{eq:eq50}) that the function $x\mapsto H(x)\Phi_{1}(y)$ belongs
to $\mathcal{V}$. Taking into account that $f$ and $H$ are linearly
independent, we get $\Phi_{1}(y)=0$. So, $y$ being arbitrary, we
obtain $\Phi_{1}=0$. Hence, using (\ref{eq:eq49}), we get that
\begin{equation}\label{eq:eq51}H=\lambda\,M-\lambda\,m.\end{equation}
Substituting this back into (\ref{eq:eq31}) we get, by an elementary
computation, that
\begin{equation}\label{eq:eq52}\begin{split}f(xy)-\lambda^{2}\,M(xy)&=(f(x)-\lambda^{2}\,M(x))m(y)+m(x)(f(y)-\lambda^{2}\,M(y))\\
&+\lambda^{2}\,m(xy)+\psi(x,y),\end{split}\end{equation} for all
$x,y\in G$.
We conclude from (\ref{eq:eq29}), (\ref{eq:eq30}), (\ref{eq:eq51}) and(\ref{eq:eq52}) that the result (2) of Lemma 3.4 holds.\\
\underline{Case A.2.2}: $\gamma=0$. Let $y\in G$ be arbitrary. The
identity (\ref{eq:eq44}) implies that $k=\Phi_{1}$. Hence we derive
from (\ref{eq:eq43}) that
$$H(x)\Phi_{1}(y)=H(y)\Phi_{1}(x)-\Phi(x,y),$$ for all $x\in G$.
Since the function $x\mapsto\Phi(x,y)$ belongs to $\mathcal{V}$ we
get, taking the identity above and (\ref{eq:eq45}) into account,
that the function $x\mapsto H(x)\Phi_{1}(y)$ belongs to
$\mathcal{V}$. As $f$ and $H$ are linearly independent modulo
$\mathcal{V}$ we infer that $\Phi_{1}(y)=0$. So, $y$ being
arbitrary, we get that $\Phi_{1}=0$. Hence the identity
(\ref{eq:eq46}) becomes
\begin{equation}\label{eq:eq53}H(xy)=m(x)H(y)+H(x)m(y)+\Phi(x,y).\end{equation} On the other hand, by using
(\ref{eq:eq37}), (\ref{eq:eq39}) and (\ref{eq:eq41}) we derive,
using the same notations, that there exist $\eta_{i}\in\mathbb{C}$
with
$i=1,\cdot\cdot\cdot,7$ such that\\
$\Phi(x,y)=\eta_{1}\,\psi(x,y)+\eta_{2}\,m(x)\psi(y,z_{1})+\eta_{3}\,m(x)\psi(y,z_{2})+\eta_{4}\,\psi(x,yz_{1})+\eta_{5}\,\psi(x,yz_{2})$
$+\eta_{6}\,\psi(xy,z_{1})+\eta_{7}\,\psi(xy,z_{2})\quad\quad\quad\quad\quad\quad\quad\quad\quad\quad\quad\quad\quad\quad$\\
 $x,y\in G$. We get that
\begin{equation}\label{eq:eq54}\begin{split}\Phi(x,y)&=\eta_{1}\,\psi(x,y)+\eta_{2}\,m(x)L_{1}(y)+\eta_{3}\,m(x)L_{2}(y)+\eta_{4}\,\psi(x,l_{1}(y))\\
&+\eta_{5}\,\psi(x,l_{2}(y))+\eta_{6}\,L_{1}(xy)+\eta_{7}\,L_{2}(xy)\end{split}\end{equation}
for all $x,y\in G$, where
$$L_{i}(x):=\psi(x,z_{i})$$ for $i=1,2$ and for all $x\in G$, and
$l_{i}:G\rightarrow G$ is defined for $i=1,2$ by $l_{i}(x)=xz_{i}$
for all $x\in G$. Hence we get from (\ref{eq:eq53}) and
(\ref{eq:eq50}) the identity
\begin{equation}\label{eq:eq55}\begin{split}&H(xy)-m(x)H(y)-H(x)m(y)=\eta_{1}\,\psi(x,y)+\eta_{2}\,m(x)L_{1}(y)+\eta_{3}\,m(x)L_{2}(y)\\
&+\eta_{4}\,\psi(x,l_{1}(y))+\eta_{5}\,\psi(x,l_{2}(y))+\eta_{6}\,L_{1}(xy)+\eta_{7}\,L_{2}(xy)\end{split}\end{equation}
for all $x,y\in G$.\\ We conclude from (\ref{eq:eq29}),
(\ref{eq:eq30}),
(\ref{eq:eq31}) and (\ref{eq:eq55}) that the result (3) of Lemma 3.4 holds.\\
\underline{Case B}: $f,g$ and $h$ are linearly independent modulo
$\mathcal{V}$. Then, according to Lemma 3.1, there exist two
functions $\varphi_{1},\varphi_{2}\in\mathcal{V}$ satisfying
(\ref{eq:eq4}), where $\psi$ is the function defined in
(\ref{eq:eq5}). Let $y\in G$ be arbitrary. Since the functions
$x\mapsto\psi(x,y)$ and $x\mapsto f(xy)-f(yx)$ belong to
$\mathcal{V}$ by assumption, so does the function
$x\mapsto\psi(y,x)$. Seeing that
$\psi(y,x)=\varphi_{1}(y)f(x)+\varphi_{2}(y)h(x)$, and that $f$ and
$h$ are linearly independent modulo $\mathcal{V}$, we get that
$\varphi_{1}(y)=\varphi_{2}(y)=0$. So, $y$ being arbitrary, we
deduce that $\psi(x,y)=0$ for all $x,y\in G$. Then the result (4) of
Lemma 3.4 holds. This completes the proof of Lemma 3.4.\end{proof}
\section{Stability of equation (\ref{eq:eq1}) on amenable groups}
Throughout this section $G$ is an amenable group with an identity
element that we denote $e$. We will extend the Sz\'{e}kelyhidi's
results \cite[Theorem 2.3]{Székelyhidi3}, about the stability of the
functional equation (\ref{eq:eq2}), to the functional equation
(\ref{eq:eq1}).
\begin{thm} Let $f,g,h:G\rightarrow\mathbb{C}$ be functions. The
function
\begin{equation*}(x,y)\mapsto f(xy)-f(x)g(y)-g(x)f(y)-h(x)h(y)\end{equation*}
is bounded if and only if one of the following assertions holds:\\
(1) $f=0$, $g$ is arbitrary and $h\in\mathcal{B}(G)$;\\
(2) $f,g,h\in\mathcal{B}(G)$;\\
(3)\[ \left\{
\begin{array}{r c l}
f&=&a\,m+\varphi,\quad\quad\quad\quad\quad\quad\quad\quad\quad\quad\quad\quad\quad\quad\quad\quad\quad\quad\quad\quad\quad\quad\quad\quad\quad\\
g&=&(1-\frac{\lambda^{2}}{2}\,a)m-\lambda\,b-\frac{\lambda^{2}}{2}\,\varphi,\quad\quad\quad\quad\quad\quad\quad\quad\quad\quad\quad\quad\quad\quad\quad\quad\quad\quad\\
h&=&\lambda\,a\,m+b+\lambda\,\varphi,\quad\quad\quad\quad\quad\quad\quad\quad\quad\quad\quad\quad\quad\quad\quad\quad\quad\quad\quad\quad\quad\quad
\end{array}
\right.
\]
where $\lambda\in\mathbb{C}$ is a constant,
$a:G\rightarrow\mathbb{C}$ is an additive function,
$m:G\rightarrow\mathbb{C}$ is a bounded multiplicative function and
$b,\varphi:G\rightarrow\mathbb{C}$ are two bounded functions;\\
(4)\[ \left\{
\begin{array}{r c l}
f&=&\alpha\,m-\alpha\,b,\quad\quad\quad\quad\quad\quad\quad\quad\quad\quad\quad\quad\quad\quad\quad\quad\quad\quad\quad\quad\quad\quad\quad\quad\\
g&=&\frac{1-\alpha\lambda^{2}}{2}\,m+\frac{1+\alpha\lambda^{2}}{2}\,b-\lambda\,\varphi,\quad\quad\quad\quad\quad\quad\quad\quad\quad\quad\quad\quad\quad\quad\quad\quad\\
h&=&\alpha\lambda\,m-\alpha\lambda\,b+\varphi,\quad\quad\quad\quad\quad\quad\quad\quad\quad\quad\quad\quad\quad\quad\quad\quad\quad\quad\quad\quad\quad
\end{array}
\right.
\]
where $\alpha,\lambda\in\mathbb{C}$ are two constants,
$m:G\rightarrow\mathbb{C}$ is a multiplicative function and
$b,\varphi:G\rightarrow\mathbb{C}$ are two bounded functions;\\
(5)\[ \left\{
\begin{array}{r c l}
f&=&f_{0},\quad\quad\quad\quad\quad\quad\quad\quad\quad\quad\quad\quad\quad\quad\quad\quad\quad\quad\quad\quad\quad\quad\quad\quad\quad\quad\\
g&=&g=g_{0}-\frac{\lambda^{2}}{2}\,f_{0}-\lambda\,b,\quad\quad\quad\quad\quad\quad\quad\quad\quad\quad\quad\quad\quad\quad\quad\quad\quad\quad\quad\\
h&=&\lambda\,f_{0}+b,\quad\quad\quad\quad\quad\quad\quad\quad\quad\quad\quad\quad\quad\quad\quad\quad\quad\quad\quad\quad\quad\quad\quad
\end{array}
\right.
\]
where $\lambda\in\mathbb{C}$ is a constant,
$b:G\rightarrow\mathbb{C}$ is a bounded function and
$f_{0},g_{0}:G\rightarrow\mathbb{C}$ are functions satisfying the
sine functional equation
$$f_{0}(xy)=f_{0}(x)g_{0}(y)+g_{0}(x)f_{0}(y),\,\,x,y\in G;$$\\
(6)\[ \left\{
\begin{array}{r c l}
f&=&-\lambda^{2}\,f_{0}+\lambda^{2}\,b,\quad\quad\quad\quad\quad\quad\quad\quad\quad\quad\quad\quad\quad\quad\quad\quad\quad\quad\quad\quad\quad\quad\\
g&=&\frac{1+\rho^{2}}{2}\,f_{0}+\rho\,g_{0}+\frac{1-\rho^{2}}{2}\,b,\quad\quad\quad\quad\quad\quad\quad\quad\quad\quad\quad\quad\quad\quad\quad\\
h&=&\lambda\rho\,f_{0}+\lambda\,g_{0}-\lambda\rho\,b,\quad\quad\quad\quad\quad\quad\quad\quad\quad\quad\quad\quad\quad\quad\quad\quad\quad\quad\quad
\end{array}
\right.
\]
where $\rho\in\mathbb{C},\,\lambda\in\mathbb{C}\setminus\{0\}$ are
two constants, $b:G\rightarrow\mathbb{C}$ is a bounded function and
$f_{0},g_{0}:G\rightarrow\mathbb{C}$ are functions satisfying the
cosine functional equation
$$f_{0}(xy)=f_{0}(x)f_{0}(y)-g_{0}(x)g_{0}(y),\,\,x,y\in G;$$
(7)\[ \left\{
\begin{array}{r c l}
f&=&\lambda^{2}\,M+am+b,\quad\quad\quad\quad\quad\quad\quad\quad\quad\quad\quad\quad\quad\quad\\
g&=&\beta\lambda(1-\frac{1}{2}\beta\lambda)M+(1-\beta\lambda)m-\frac{1}{2}\beta^{2}\,a\,m-\frac{1}{2}\beta^{2}\,b,\quad\quad\quad\quad\quad\quad\quad\\
h&=&\lambda(1-\beta\lambda)M-\lambda\,m-\beta\,a\,m-\beta\,b,\quad\quad\quad\quad\quad\quad\quad\quad\quad\quad\quad\quad
\end{array}
\right.
\]
where $\beta\in\mathbb{C},\lambda\in\mathbb{C}\setminus\{0\}$ are
tow constants, $m,M:G\rightarrow\mathbb{G}$ are two multiplicative
functions such that $m$ is bounded, $a:G\rightarrow\mathbb{C}$ is an
additive function and
$b:G\rightarrow\mathbb{C}$ is a bounded function;\\
(8)\[ \left\{
\begin{array}{r c l}
f&=&\frac{1}{2}a^{2}\,m+\frac{1}{2}a_{1}\,m+b,\quad\quad\quad\quad\quad\quad\quad\quad\quad\quad\quad\quad\quad\quad\quad\\
g&=&-\frac{1}{4}\beta^{2}\,a^{2}\,m+\beta\,a\,m-\frac{1}{4}\beta^{2}\,a_{1}\,m+m-\frac{1}{2}\beta^{2}\,b,\quad\quad\quad\quad\quad\quad\quad\quad\\
h&=&-\frac{1}{2}\beta\,a^{2}\,m+a\,m-\frac{1}{2}\beta\,a_{1}\,m-\beta\,b,\quad\quad\quad\quad\quad\quad\quad\quad\quad\quad\quad\quad\quad
\end{array}
\right.
\]
where $\beta\in\mathbb{C}$ is a constant, $m:G\rightarrow\mathbb{C}$
is a nonzero bounded multiplicative function,
$a,a_{1}:G\rightarrow\mathbb{C}$ are two additive functions such
that $a\neq0$ and
$b:G\rightarrow\mathbb{C}$ is a bounded function;\\
(9) $g=-\frac{1}{2}\beta^{2}\,f+(1+\beta\,a)m+\beta\,b$ and
$h=-\beta\,f+a\,m+b$, where $\beta\in\mathbb{C}$ is a constant and
$a:G\rightarrow\mathbb{C}$ is an additive function,
$m:G\rightarrow\mathbb{C}$ is a nonzero bounded multiplicative
function and $b:G\rightarrow\mathbb{C}$ is a bounded function such
that the function \begin{equation*}\begin{split}&(x,y)\mapsto
f(xy)m((xy)^{-1})-\frac{1}{2}a^{2}(xy)-(f(x)m(x^{-1})-\frac{1}{2}a^{2}(x))\\
&-(f(y)m(y^{-1})-\frac{1}{2}a^{2}(y))-a(x)b(y)m(y^{-1})-a(y)b(x)m(x^{-1})\end{split}\end{equation*}
is bounded;\\
(10) $f(xy)=f(x)g(y)+g(x)f(y)+h(x)h(y)$ for all $x,y\in G$.
\end{thm}
\begin{proof}First we prove the necessity. Applying the Lemma
3.3(1), Lemma 3.3(2), Lemma 3.3(4), Lemma 3.3(5), Lemma 3.4(1) and
Lemma 3.4(4) with $\mathcal{V}=\mathcal{B}(G)$ we get that either
one of the conditions (1), (2), (4), (5), (6), (10) in Theorem 4.1
is satisfied, or we have one of the following cases:\\
\underline{Case A:}
\begin{equation*}g+\frac{\lambda^{2}}{2}\,f=m-\lambda\,b\end{equation*}
and
\begin{equation*}h-\lambda\,f=b,\end{equation*}
where $\lambda\in\mathbb{C}$ is a constant,
$b:G\rightarrow\mathbb{C}$ is a bounded function and
$m:G\rightarrow\mathbb{C}$ is a bounded multiplicative function.
From (\ref{eq:eq5}) and the identities above we obtain, by an
elementary computation,
\begin{equation}\label{eq:eq56}g=-\frac{\lambda^{2}}{2}\,f+m-\lambda\,b,\end{equation}
\begin{equation}\label{eq:eq57}h=\lambda\,f+b\end{equation}
and
\begin{equation}\label{eq:eq58}f(xy)-f(x)m(y)-m(x)f(y)=\psi(x,y)+b(x)b(y)\end{equation}
for all $x,y\in G$. If $m\neq0$ then, by multiplying both sides of
(\ref{eq:eq58}) by $m((xy)^{-1})$, and using the fact that $m$ is a
bounded multiplicative function, and that the functions $b$ and
$\psi$ are bounded, we get that the function $(x,y)\mapsto
f(xy)m((xy)^{-1})-f(x)m(x^{-1})-f(y)m(y^{-1})$ is bounded. Notice
that we have the same result if $m=0$. So, according to Hyers's
theorem \cite[Theorem 3.1]{Székelyhidi1}, there exist an additive
function $a:G\rightarrow\mathbb{C}$ and a function
$\varphi_{0}\in\mathcal{B}(G)$ such that
$f(x)m(x^{-1})-a(x)=b_{0}(x)$ for all $x\in G$. Then, by putting
$\varphi=m\,\varphi_{0}$, we get that $f=a\,m+\varphi$ with
$\varphi\in\mathcal{B}(G)$. Substituting this back into
(\ref{eq:eq56}) and (\ref{eq:eq57}) we obtain, by an elementary
computation, that
$g=(1-\frac{\lambda^{2}}{2}\,a)m-\lambda\,b-\frac{\lambda^{2}}{2}\,\varphi$
and $h=\lambda\,a\,m+b+\lambda\,\varphi$. So the result (3)
of Theorem 4.1 holds.\\
\underline{Case B:}
\begin{equation*}\begin{split}f(xy)-\lambda^{2}\,M(xy)&=(f(x)-\lambda^{2}\,M(x))m(y)+m(x)(f(y)
-\lambda^{2}\,M(y))\\&+\lambda^{2}\,m(xy)+\psi(x,y)\end{split}\end{equation*}
for all $x,y\in G,$$$g=\frac{1}{2}\beta^{2}\,f+\beta\,h+m$$ and
$$\beta\,f+h=\lambda\,M-\lambda\,m,$$ where
$\beta\in\mathbb{C},\,\lambda\in\mathbb{C}\setminus\{0\}$ are
constants, $m,M:G\rightarrow\mathbb{C}$ are multiplicative functions
such that $m\in\mathcal{B}(G)$, $M\not\in\mathcal{B}(G)$ and $\psi$
is the function defined in (\ref{eq:eq5}). If $m\neq0$ then, by
multiplying both sides of the first identity above by $m((xy)^{-1})$
and using that $m$ is multiplicative, we get that
\begin{equation*}\begin{split}&(f(xy)-\lambda^{2}\,M(xy))m((xy)^{-1})\\&=(f(x)-\lambda^{2}\,M(x))m(x^{-1})+(f(y)
-\lambda^{2}\,M(y))m(y^{-1})+\lambda^{2}+m((xy)^{-1})\psi(x,y)\end{split}\end{equation*}
for all $x,y\in G$. Since the functions $m$ and $\psi$ are bounded,
then we get from the identity above that the function
\begin{equation*}\begin{split}&(x,y)\mapsto
(f(xy)-\lambda^{2}\,M(xy))m((xy)^{-1})-(f(x)-\lambda^{2}\,M(x))m(x^{-1})\\&-(f(y)
-\lambda^{2}\,M(y))m(y^{-1})\end{split}\end{equation*} is bounded.
Notice that we have the same result if $m=0$. So, according to
Hyers's theorem \cite[Theorem 3.1]{Székelyhidi1}, there exist an
additive function $a:G\rightarrow\mathbb{C}$ and a function
$b_{0}\in\mathcal{B}(G)$ such that
$$(f(x)-\lambda^{2}\,M(x))m(x^{-1})-a(x)=b_{0}(x)$$ for all $x\in
G$. Then, by putting $b=m\,b_{0}$. we derive that
$$f=\lambda^{2}\,M+a\,m+b$$ with $b\in \mathcal{B}(G)$. As
$g=\frac{1}{2}\beta^{2}\,f+\beta\,h+m$ and
$\beta\,f+h=\lambda\,M-\lambda\,m,$ we obtain
\begin{equation*}\begin{split}h&=-\beta(\lambda^{2}\,M+a\,m+b)+\lambda\,M-\lambda\,m\\
&=\lambda(1-\beta\lambda)M-\lambda\,m-\beta\,a\,m-\beta\,b\end{split}\end{equation*}
and
\begin{equation*}\begin{split}g&=\frac{1}{2}\beta^{2}(\lambda^{2}\,M+a\,m+b)+\beta(\lambda(1-\beta\lambda)M-\lambda\,m-\beta\,a\,m-\beta\,b)+m\\
&=\beta\lambda(1-\frac{1}{2}\beta\lambda)M+(1-\beta\lambda)m-\frac{1}{2}\beta^{2}\,a\,m-\frac{1}{2}\beta^{2}\,b.
\end{split}\end{equation*}
The result occurs in (7) of Theorem 4.1.\\
\underline{Case C:}
\begin{equation*}f(xy)=f(x)m(y)+m(x)f(y)+H(x)H(y)+\psi(x,y),\end{equation*}
\begin{equation*}\begin{split}&H(xy)-H(x)m(y)-m(x)H(y)=\eta_{1}\,\psi(x,y)+\eta_{2}\,m(x)L_{1}(y)+\eta_{3}\,m(x)L_{2}(y)\\
&+\eta_{4}\,\psi(x,l_{1}(y))+\eta_{5}\,\psi(x,l_{2}(y))+\eta_{6}\,L_{1}(xy)+\eta_{7}\,L_{2}(xy)\end{split}\end{equation*}
for all $x,y\in G$,\\
\begin{equation*}g=\frac{1}{2}\beta^{2}\,f+\beta\,h+m\end{equation*}
and
\begin{equation*}H=\beta\,f+h\end{equation*}
and where $\beta,\eta_{1},\cdot\cdot\cdot,\eta_{7}\in\mathbb{C}$ are
constants, $m:G\rightarrow\mathbb{C}$ is a bounded multiplicative
function, $L_{1},L_{2}\in\mathcal{B}(G)$, $l_{1},l_{2}:G\rightarrow
G$ are mappings, and $\psi$ is the function defined in (\ref{eq:eq5}).\\
If $H\in \mathcal{B}(G)$ then $f$ and $h$ are linearly dependent
modulo $\mathcal{B}(G)$. So, according to Lemma 3.3, on of the
assertions
(1)-(5) of Theorem 4.1 holds.\\
In what follows we assume that $H\not\in\mathcal{B}(G)$. Since the
functions $m$, $L_{1}$, $L_{2}$ and $\psi$ are bounded, we get from
the above second identity that the function $$(x,y)\mapsto
H(xy)-H(x)m(y)-m(x)H(y)$$ is bounded. Hence $m\neq0$ because
$H\not\in\mathcal{B}(G)$. Then, according to \cite[Theorem
2.3]{Székelyhidi3} and taking the assumption on $H$ into account, we
have one
of the following subcases:\\
\underline{Subcase C.1:} $H=a\,m+b$, where
$a:G\rightarrow\mathbb{C}$ is additive and $b\in\mathcal{B}(G)$.
Then $\beta\,f+h=a\,m+b$, which implies that
$$h=-\beta\,f+a\,m+b.$$ Moreover, since
$g=\frac{1}{2}\beta^{2}\,f+\beta\,h+m$ we get that
$$g=-\frac{1}{2}\beta^{2}\,f+m+\beta\,a\,m+\beta\,b.$$
Let $x,y\in G$ be arbitrary. By using the first identity in the
present case, we get that
\begin{equation*}\begin{split}&\psi(x,y)=f(xy)-f(x)m(y)-m(x)f(y)-(a(x)m(x)+b(x))(a(y)m(y)+b(y))\\
&=f(xy)-f(x)m(y)-m(x)f(y)-a(x)a(y)m(xy)-m(x)a(x)b(y)\\
&-m(y)a(y)b(x)-b(x)b(y).\end{split}\end{equation*} Since $m$ is a
nonzero multiplicative function on the group $G$ we have
$m(xy)=m(x)m(y)\neq0$ and
$m((xy)^{-1})=m(x^{-1})m(y^{-1})=(m(x))^{-1}(m(y))^{-1}$. Hence, by
multiplying both sides of the identity above we get that
\begin{equation*}\begin{split}&m((xy)^{-1})[\psi(x,y)b(x)b(y)]=f(xy)m((xy)^{-1})-f(x)m(x^{-1})-f(y)m(y^{-1})\\
&-a(x)a(y)-a(x)b(y)m(y^{-1})-a(y)b(x)m(x^{-1})\\
&=(f(xy)m((xy)^{-1})-\frac{1}{2}a^{2}(xy))-(f(x)m(x^{-1})-\frac{1}{2}a^{2}(x))
\\&-(f(y)m(y^{-1})-\frac{1}{2}a^{2}(y))-a(x)b(y)m(y^{-1})-a(y)b(x)m(x^{-1}).\end{split}\end{equation*}
So, $x$ and $y$ being arbitrary, and the functions $m$, $b$ and
$\psi$ are bounded, we deduce that the function
\begin{equation*}\begin{split}&(x,y)\mapsto
f(xy)m((xy)^{-1})-\frac{1}{2}a^{2}(xy)-(f(x)m(x^{-1})-\frac{1}{2}a^{2}(x))\\
&-(f(y)m(y^{-1})-\frac{1}{2}a^{2}(y))-a(x)b(y)m(y^{-1})-a(y)b(x)m(x^{-1})\end{split}\end{equation*}
is bounded. The result occurs in (9) of the list of Theorem 4.1.\\
\underline{Subcase C.2:} $H(xy)=H(x)m(y)+H(y)m(x)$ for all $x,y\in
G$. Since $m$ is a nonzero multiplicative function on the group $G$
we have $m(x)\neq0$ for all $x\in G$. Then, in view of
$H\not\in\mathcal{B}(G)$, we get from the last functional equation
that there exists a nonzero additive function
$a:G\rightarrow\mathbb{C}$ such that $H=a\,m$. Substituting this
back in the first identity in the present case and proceeding
exactly as in Subcase C.1, we get that the function
\begin{equation*}\begin{split}&(x,y)\mapsto
2f(xy)m((xy)^{-1})-a^{2}(xy)-(2f(x)m(x^{-1})-a^{2}(x))\\
&-(2f(y)m(y^{-1})-a^{2}(y))\end{split}\end{equation*} is bounded.
Hence, according to Hyers's theorem \cite[Theorem
3.1]{Székelyhidi1}, there exist an additive function
$a_{1}:G\rightarrow\mathbb{C}$ and a function
$b_{0}\in\mathcal{B}(G)$ such that
$2f(x)m(x^{-1})-a^{2}(x)=a_{1}(x)+b_{0}(x)$ for all $x,y\in G$. So,
by putting $b=\frac{1}{2}m\,b_{0}$ we deduce that
$b\in\mathcal{B}(G)$ because $m,b_{0}\in\mathcal{B}(G)$ and
\begin{equation}\label{eq:eq59}f=\frac{1}{2}a^{2}\,m+\frac{1}{2}a_{1}\,m+b.\end{equation}Since $H=\beta\,f+h$ and
$g=\frac{1}{2}\beta^{2}\,f+\beta\,h+m$ we get, by using
(\ref{eq:eq59}) and an elementary computation, that
$g=-\frac{1}{4}\beta^{2}\,a^{2}\,m+\beta\,a\,m-\frac{1}{4}\beta^{2}\,a_{1}\,m+m-\frac{1}{2}\beta^{2}\,b$
and
$h=-\frac{1}{2}\beta\,a^{2}\,m+a\,m-\frac{1}{2}\beta\,a_{1}\,m-\beta\,b.$
The result occurs in (8) of the list of Theorem 4.1.\\
Conversely, we check by elementary computations that if one of the
assertions (1)-(10) in Theorem 4.1 is satisfied then the function
$(x,y)\mapsto f(xy)-f(x)g(y)-g(x)f(y)-h(x)h(y)$ is bounded. This
completes the proof of Theorem 4.1.
\end{proof}
\textbf{Acknowledgments.} The authors are grateful to referees for
the thorough review of this paper.

\end{document}